\documentclass[letterpaper, 10 pt, conference]{ieeeconf}  

\IEEEoverridecommandlockouts                              
\usepackage{graphicx}          

\usepackage{textcomp}

\usepackage{amsmath}
\usepackage{amssymb,color}
\usepackage{amsfonts}
\usepackage{subcaption}

\usepackage{algorithm}

\usepackage{algpseudocode}

\usepackage{mathtools}

\let\labelindent\relax 
\usepackage{enumitem}

\usepackage[dvipsnames]{xcolor}
\usepackage{tikz}
\usetikzlibrary{3d,calc}
\usepackage{pgfplots}
\usepackage{bm,comment}
\usepgfplotslibrary{fillbetween}
\usetikzlibrary{arrows,hobby,tikzmark}
\usetikzlibrary{decorations.markings,decorations.pathreplacing}
\usetikzlibrary{math}
\tikzstyle{block} = [draw, rectangle, 
    minimum height=2em, minimum width=2em]
\tikzstyle{sum} = [draw, circle, node distance=1cm]
\tikzstyle{input} = [coordinate]
\tikzstyle{output} = [coordinate]
\tikzstyle{pinstyle} = [pin edge={to-,thin,black}]

\usepackage{amsthm}
\newtheoremstyle{customthm}
  {1.5mm}
  {1.5mm}
  {\itshape}
  {}
  {\bfseries}
  {.}
  {1.5mm}
  {}
  
\theoremstyle{customthm}
\newtheorem{Theorem}{Theorem}

\newtheorem{Lemma}{Lemma}

\newtheorem{Remark}{Remark}
\newtheorem{Corollary}{Corollary}

\title{\LARGE \bf
A Dual System-Level Parameterization for Identification from Closed-Loop Data
}
\author{Amber Srivastava, Mingzhou Yin, Andrea Iannelli, Roy S. Smith, \IEEEmembership{Fellow, IEEE}
\thanks{This work was supported by the NCCR Automation Grant 180545 funded by the Swiss National Science Foundation.}
\thanks{Amber Srivastava, Mingzhou Yin, and Roy S. Smith are with the Automatic Control Laboratory, Swiss Federal Institute of Technology (ETH Z\"urich), 8092 Zurich, Switzerland {\tt\small(e-mail:
asrivastava/myin/rsmith@control.ee.ethz.ch)}.}%
\thanks{Andrea Iannelli is with the University of Stuttgart, Institute for Systems Theory and Automatic Control, 70569 Stuttgart, Germany {\tt\small andrea.iannelli@ist.uni-stuttgart.de}}%
}

\begin{document}
\maketitle
\thispagestyle{empty}
\pagestyle{empty}

\begin{abstract}
This work presents a dual system-level parameterization (D-SLP) method for closed-loop system identification. The recent system-level synthesis framework parameterizes all stabilizing controllers via linear constraints on closed-loop response functions, known as system-level parameters. It was demonstrated that several structural, locality, and communication constraints on the controller can be posed as convex constraints on these system-level parameters. In the current work, the identification problem is treated as a {\em dual} of the system-level synthesis problem. The plant model is identified from the dual system-level parameters associated to the plant. In comparison to existing closed-loop identification approaches (such as the dual-Youla parameterization), the D-SLP framework neither requires the knowledge of a nominal plant that is stabilized by the known controller, nor depends upon the choice of factorization of the nominal plant and the stabilizing controller. Numerical simulations demonstrate the efficacy of the proposed D-SLP method in terms of identification errors, compared to existing closed-loop identification techniques. 
\end{abstract}

\section{Introduction}\label{sec: Introduction}

System identification estimates plant models from observed input-output data \cite{LjungBook2}. The majority of system identification research assumes that the data are generated in open loop, i.e., the inputs are independently selected. However, in many applications, open-loop data collection is not possible due to instability or operational constraints \cite{Gustavsson_1977}. Although open-loop methods can be directly applied to closed-loop data (known as direct identification), the performance is often not satisfactory due to the correlation between the input and the disturbance/noise \cite{VANDENHOF1998173}. Furthermore, the identified models are not guaranteed to be stabilized by the closed-loop controller, which makes the model practically unusable.

Therefore, specific algorithms for closed-loop identification are required. Instead of directly identifying the plant, these methods typically first identify some parametrizations of the plant and recover the plant model indirectly. The classical idea for such indirect identification is to first estimate the transfer function from the reference/ external signals $r_1[t],r_2[t]$ to the output $y[t]$ represented in the Figure \ref{fig:sysID_Rep}, from which the plant model $\mathbf{G}$ is calculated algebraically \cite{van1993indirect}. When the controller is unknown, this method is extended to the input-output method where the transfer function from the reference to the input $u[t]$ is also identified \cite[Section~13.5]{LjungBook2}. A more general approach, known as the coprime factor identification \cite{schrama1991open,Van_den_Hof_1995}, solves two open-loop problems with a filtered reference, where the filter is designed using a coprime factorization of an initial plant estimate.

\begin{figure}
 \centering	
 \begin{tikzpicture}[auto,>=latex']
    \node [input, name=Dr1] {};
    \node [sum, left of=Dr1,scale=0.4] (sum1) {$\text{\Huge-}$};
    \node [block, left of=sum1, node distance=1.4cm] (K) {$\mathbf{K}$};
    \node [sum, left of=K,node distance=1.5cm,scale=0.4] (sum2) {$\text{\Huge +}$};
    \node [input, above of=sum2,node distance=1.2cm] (Dr2) {};
    \node [block, left of=sum2,node distance=1.5cm] (G) {$\mathbf{G}$};
    \node [sum, left of=G,node distance=1.4cm,scale=0.4] (sum3) {$\text{\Huge +}$};
    \node [block, above of=sum3,node distance=1cm] (S) {$\mathbf{S}$};
    \node [input, above of=S,node distance=1.4cm] (De) {};
    \node [input, above of=S, name=De] {};
    \node [input, left of=sum3,node distance=0.8cm] (T1) {};
    \node [input, left of=T1,node distance=0.8cm] (T1_1) {};
    \node [input, below of=T1,node distance=1.4cm] (T2) {};
    \node [input, below of=sum1,node distance=1.4cm] (T3) {};
 
    \draw [->,anchor=south] (Dr1) -- node {$r_1[t]$} (sum1);
    \draw [->,anchor=south] (sum1) -- node {$y[t]$} (K);
    \draw [->,anchor=north] (K) -- node {$u[t]$} (sum2);
    \draw [->,anchor=south] (sum2) -- node {$\bar{u}[t]$} (G);
    \draw [->] (Dr2) -- node{$r_2[t]$} (sum2);
    \draw [->] (sum2) -- (G);
    \draw [->] (G) -- node{$\bar{y}[t]$} (sum3);
    \draw [->] (S) -- (sum3);
    \draw [->] (De) -- node {$e(k)$} (S);
    \draw [-] (sum3) -- (T1);
    \draw [->] (T1) -- node{$\hat{y}[t]$} (T1_1);
    \draw[-] (T1) -- (T2);
    \draw[-] (T2) -- (T3);
    \draw[->,anchor= south west] (T3) -- node{$+$} (sum1);
\end{tikzpicture}
     \caption[something]{Closed-loop configuration for identification.\footnotemark}
    \label{fig:sysID_Rep}
\end{figure}
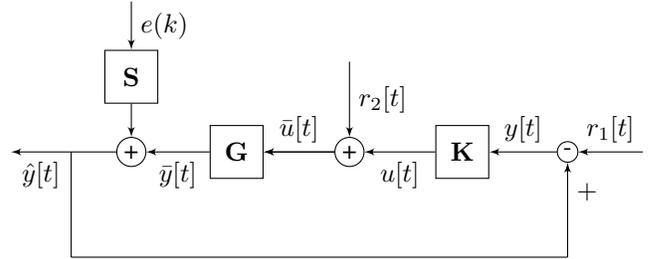
\footnotetext{Without loss of generality, we consider the configuration in which the input to the controller is $\hat{y}[t]-r_1[t]$ and not the traditional $r_1[t]-\hat{y}[t]$. This is for notational convenience in the later sections.}

To guarantee the closed-loop stability of the identified plant, the dual-Youla method is proposed \cite{hansen1988fractional,4790411}. This method originates from the Youla parametrization which parameterizes all the stabilizing controllers for control design \cite[Section~4.8]{skogestad2005multivariable}. The dual parametrization given by reversing the role of the controller and the plant, provides a parametrization of all plants stabilized by a known controller. The dual-Youla parameter can then be identified by an open-loop problem with filtered input, output, and reference, where the filter depends on a coprime factorization of the controller in addition to that of an initial plant estimate (also referred to as the nominal model).

In this work, we apply a similar dual approach to the system-level synthesis framework, recently developed in \cite{wang2019system}. In this framework, the controller is parameterized in terms of closed-loop response functions, also known as the system-level parameters, which map the process and measurement disturbances to the control actions and the states. The work done in \cite{wang2019system} illustrates the benefit of such a controller parameterization in terms of incorporating structural, locality, and communication constraints, which occur naturally in the context of decentralized controller design for large-scale cyber-physical systems, as convex constraints in the synthesis problem \cite{wang2018separable}. In the current work, we present a dual system-level parameterization (D-SLP) approach to identification of the plant transfer function from closed-loop data, where the dual parameters characterize all possible transfer functions that are stabilized by a given controller. These dual parameters are learned from the available closed-loop data, based on which the plant transfer function is estimated.

A straightforward benefit of the D-SLP methodology over approaches such as the indirect or input-output identification methods is that it guarantees that the identified plant transfer function is stabilized by the known controller - similar to the dual-Youla parameterization \cite{hansen1988fractional}. In addition to that, and contrary to the dual-Youla and the co-prime factorization approaches, the D-SLP identification scheme is independent of the choice of the nominal plant and its factorization, as well as the factorization of the controller, which makes it a tuning-free approach to identification in the closed-loop setting. Our simulations demonstrate that the D-SLP approach performs better than the dual-Youla and the co-prime identification methods for different choices of the nominal plants (that are required by the latter two schemes), including a practical two-stage scheme to select the nominal plant. 

{\em Notations and problem setting:} Lower case 
(e.g., $x$) and bold lower case (e.g. $\mathbf{x}$) Latin letters denote vectors and signals, respectively. Upper case Latin letters (such as $A$) denote matrices, and bold upper case letters (for example $\mathbf{G}$) denote transfer functions. The set $\mathcal{RH}_{\infty}$ comprises of all the stable proper transfer function matrices. For the purpose of system identification, we consider the closed-loop system configuration illustrated in Figure \ref{fig:sysID_Rep}. Here, $\mathbf{G}$ is a discrete-time linear time-invariant (LTI) system (possibly unstable) that is stabilized by the LTI controller $\mathbf{K}$, and $\mathbf{S}$ is a stable filter for the independent and identically distributed noise $e[t]$. It is assumed that the measurements of the output $y[t]$ and the reference signal $r_1[t]$ and $r_2[t]$,  as well as the knowledge of the controller $\mathbf{K}$ are available.

The paper is organized as follows. Section \ref{sec: SLP} gives a brief introduction to the system-level parameterization. In Section \ref{sec: D_SLP}, we present the D-SLP of all the plants that are stabilized by a known controller. Section \ref{sec: SysID_D_SLP} develops the D-SLP methodology to identify the plant from closed-loop data, and Section \ref{sec: Simulations} compares our identification scheme with the existing closed-loop identification strategies. We conclude the paper in Section \ref{sec: Conclusions}. 

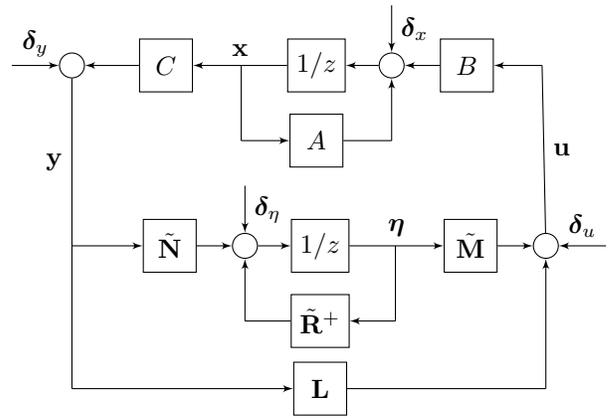
\begin{figure}
 \centering	
 \begin{tikzpicture}[auto,>=latex']
    \node [input,name=Dy] {};
    \node [sum,right of=Dy,node distance=0.8cm] (sum1) {};
    \node [block,right of=sum1,node distance=1.25cm] (C2){$C$};
    \node [input,right of=C2,node distance=1cm] (T1){};
    \node [block,right of=T1,node distance=1cm] (bYz){$1/z$};
    \node [input,below of=T1,node distance=1cm] (T1_1){};
    \node [block,right of=T1_1,node distance=1cm] (A){$A$};
    \node [input,right of=A,node distance=1cm] (T1_2){};
    \node [sum,right of=bYz,node distance=1cm] (sum2) {};
    \node [input,above of=sum2,node distance=0.8cm] (Dx) {};
    \node [block,right of=sum2,node distance=1cm] (B2) {$B$};
    \node [input,right of=B2,node distance=1cm] (T2) {};
    \draw[->] (Dy)-- node{$\bm\delta_y$} (sum1);
    \draw[->] (C2) -- (sum1);
    \draw[->,anchor=south] (bYz) -- node{$\mathbf{x}$} (C2);
    \draw[->] (sum2) -- (bYz);
    \draw[->] (Dx) -- node{$\bm\delta_x$} (sum2);
    \draw[->] (B2) -- (sum2);
    \draw[->] (T2) -- (B2);
    \draw[-] (T1) -- (T1_1);
    \draw[->] (T1_1) -- (A);
    \draw[-] (A) -- (T1_2);
    \draw[->] (T1_2) -- (sum2);
    
    \node [input,below of=sum1,node distance=2.4cm] (Dn) {};
    \node [block,right of=Dn,node distance=1.3cm] (N) {$\tilde{\mathbf{N}}$};
    \node [sum,right of=N,node distance=1cm] (sum3){};
    \node [block,right of=sum3,node distance=1cm] (bYz2){$1/z$};
    \node [input,below of=sum3,node distance=1cm] (T1_11){};
    \node [block,right of=T1_11,node distance=1cm] (R){$\tilde{\mathbf{R}}^+$};
    \node [input,right of=R,node distance=1cm] (T1_21){};
    \node [input,right of=bYz2,node distance=1cm] (T1_22) {};
    \node [block,right of=T1_22,node distance=1cm] (M) {$\tilde{\mathbf{M}}$};
    \node [sum,right of=M,node distance=1cm] (sum4) {};
    \node [input,right of=sum4,node distance=0.8cm] (Du) {};
    \node [input,above of=sum3,node distance=0.8cm] (Db) {};
    \draw[->] (Db) -- node{$\bm\delta_{\eta}$} (sum3);
    \draw[->] (Dn)-- (N);
    \draw[->] (N) -- (sum3);
    \draw[->,anchor=south] (sum3) -- (bYz2);
    \draw[->] (bYz2) -- node{$\bm\eta$} (M);
    \draw[->] (M) -- (sum4);
    \draw[->] (T1_11) -- (sum3);
    \draw[-] (R) -- (T1_11);
    \draw[->] (T1_21) -- (R);
    \draw[-] (T1_22) -- (T1_21);
    \draw[->,anchor=south] (Du) -- node{$\bm\delta_u$} (sum4);
    \draw[-,anchor=west] (sum4) -- node{$\mathbf{u}$} (T2);
    \draw[-,anchor=east] (sum1) -- node{$\mathbf{y}$} (Dn);

    \node [input,below of=Dn,node distance=1.9cm] (T3) {};
    \node [block,right of=T3,node distance=3.3cm] (L) {$\mathbf{L}$};
    \node [input,right of=L,node distance=3cm] (T4) {};
    \draw[-] (Dn) -- (T3);
    \draw[->] (T3) -- (L);
    \draw[-] (L) -- (T4);
    \draw[->] (T4) -- (sum4);
        
\end{tikzpicture}
     \caption{Output feedback controller with $\tilde{\mathbf{R}}^+=z(I-z\mathbf{R})$, $\tilde{\mathbf{M}} = z\mathbf{M}$, and $\tilde{\mathbf{N}}=z\mathbf{N}$. Note that $\tilde{\mathbf{R}}^+,\tilde{\mathbf{M}},\tilde{\mathbf{N}}\in\mathcal{RH}_{\infty}$.}
    \label{fig:SyS1}
\end{figure}

\section{System-Level Parameterization}\label{sec: SLP}
In this section, we briefly review the system-level parameterization (SLP) \cite{wang2019system} developed for output feedback control design, as it forms the basis of the proposed identification approach. Let $\mathbf{P} = [\mathbf{P}_{1}~~\mathbf{P}_2]$ be an open loop plant, such that its output is given by
\begin{align}\label{eq: OutputFeedback}
\mathbf{y} = \mathbf{P}_1\mathbf{w} + \mathbf{P}_2\mathbf{u},
\end{align}
where the transfer function $\mathbf{P}_1$ filters the exogenous signals $\mathbf{w}$, and $\mathbf{P}_2$ filters the control input $\mathbf{u}$ (as in Figure \ref{fig:sysID_Rep}). Drawing parallels with the configuration illustrated in Figure \ref{fig:sysID_Rep}, we note that $\mathbf{w}^\top=[\mathbf{e}^\top~~\mathbf{r}_1^\top~~\mathbf{r}_2^\top]$, 
$\mathbf{P}_1 = [\mathbf{S}~~-I~~\mathbf{G}]$, $\mathbf{u}$ is the output from the controller $\mathbf{K}$, and $\mathbf{P}_2=\mathbf{G}$. Under the assumption that $\mathbf{P}_2$ is a strictly proper transfer function, the dynamics of the plant $\mathbf{P}$ are described as follows
\begin{subequations}\label{eq: dynamicsP}
\begin{align}
x[t+1] &= Ax[t]+Bu[t]+B_1w[t],\\
y[t] &= Cx[t] + D_1w[t],
\end{align}
\end{subequations}
where $x,u,w$, and $y$ denote the state, control input, exogenous input, and output vectors, respectively. Let $\delta_x[t]:=B_1w[t]$ be the disturbance to the plant state, and $\delta_y[t]:=D_1w[t]$ be the disturbance to the plant output. Taking the $z$-transform of \eqref{eq: dynamicsP}, and substituting the output feedback law $\mathbf{u}=\mathbf{Ky}$, we obtain the following closed-loop system responses from disturbance pair $({\bm\delta}_x,{\bm\delta}_y)$ to the state and control input signal pair $(\mathbf{x},\mathbf{u})$ as
\begin{align}\label{eq: Cl_SR}
\begin{bmatrix}
\mathbf{x}\\
\mathbf{u}
\end{bmatrix}=
\begin{bmatrix}
\mathbf{R} & \mathbf{N}\\
\mathbf{M} & \mathbf{L}
\end{bmatrix}
\begin{bmatrix}
{\bm\delta}_x\\
{\bm\delta}_y
\end{bmatrix}.
\end{align}

The following theorem from \cite{wang2019system} algebraically characterizes the set $\{\mathbf{R,M,N,L}\}$ of system responses that are achieved by the internally stabilizing controller $\mathbf{K}$, as well as parameterizes a internally stabilizing controller $\mathbf{K}$ in terms of the system responses.

\begin{Theorem}\label{thm: main}
Consider the system \eqref{eq: OutputFeedback} with output feedback $\mathbf{u} = \mathbf{Ky}$. The following statements are true.
\begin{itemize}
\item[1)] The closed-loop system responses \eqref{eq: Cl_SR} from $(\bm\delta_x,\bm\delta_y)$ to $(\mathbf{x},\mathbf{u})$ lie in the following affine subspace
\begin{subequations}\label{eq: AffSubSp}
\begin{align}
\begin{bmatrix}
zI-A & -B
\end{bmatrix}
\begin{bmatrix}
\mathbf{R} & \mathbf{N}\\
\mathbf{M} & \mathbf{L}
\end{bmatrix}&=\begin{bmatrix}
I & 0
\end{bmatrix},\\
\begin{bmatrix}
\mathbf{R} & \mathbf{N}\\
\mathbf{M} & \mathbf{L}
\end{bmatrix}\begin{bmatrix}
zI-A\\
-C
\end{bmatrix}&=
\begin{bmatrix}
I\\
0
\end{bmatrix}.\\
\mathbf{R,M,N}\in\frac{1}{z}\mathcal{RH}_{\infty}&,\ \mathbf{L}\in\mathcal{RH}_{\infty},
\end{align}
\end{subequations}
\item[2)] For the set of transfer function matrices $\{\mathbf{R,M,N,L}\}$ that lie in the affine subspace \eqref{eq: AffSubSp}, the controller $\mathbf{K=L-MR^{-1}N}$ is internally stabilizing, and can be implemented as in Figure \ref{fig:SyS1}.
\end{itemize}
\end{Theorem}

\begin{figure}
 \centering	
 \begin{tikzpicture}[auto,>=latex',scale=0.8]
    \node [input,name=Dy] {};
    \node [sum,right of=Dy,node distance=0.8cm] (sum1) {};
    \node [input,above of=sum1,node distance=0.8cm] (Hy) {};
    \node [sum,below of=sum1,node distance=0.8cm] (Dn) {};
    \node [block,right of=Dn,node distance=1.3cm] (N) {$\tilde{\mathbf{N}}$};
    \node [sum,right of=N,node distance=1cm] (sum3){};
    \node [block,right of=sum3,node distance=1cm] (bYz2){$1/z$};
    \node [input,below of=sum3,node distance=1cm] (T1_11){};
    \node [block,right of=T1_11,node distance=1cm] (R){$\tilde{\mathbf{R}}^+$};
    \node [input,right of=R,node distance=1cm] (T1_21){};
    \node [input,right of=bYz2,node distance=1cm] (T1_22) {};
    \node [block,right of=T1_22,node distance=1cm] (M) {$\tilde{\mathbf{M}}$};
    \node [sum,right of=M,node distance=1cm] (sum4) {};
    \node [input,right of=sum4,node distance=0.8cm] (Du) {};
    \node [input,below of=Du,node distance=2.8cm] (Du1) {};
    \node [input,left of=Dn,node distance=0.8cm] (Du2) {};
    \node [input,below of=Du2,node distance=2.8cm] (Du3) {};
    \node [block,left of=Du1,node distance=3.8cm] (minD22) {$-D$};
    \node [input,above of=sum3,node distance=0.8cm] (Db) {};
    \node [sum,above of=sum4,node distance=0.8cm] (sum5) {};
    \node [input,right of=sum5,node distance=0.8cm] (DelU) {};
    \node [input,above of=sum5,node distance=0.8cm] (TTyy) {};
    \draw[->] (Dy) -- node{$\bm\delta_y$} (sum1);
    \draw[->] (Hy) -- (sum1);
    \draw[->] (sum1) -- (Dn);
    \draw[->] (Db) -- node{$\bm\delta_{\eta}$} (sum3);
    \draw[->] (Dn)-- (N);
    \draw[->] (N) -- (sum3);
    \draw[->,anchor=south] (sum3) -- (bYz2);
    \draw[->] (bYz2) -- node{$\bm\eta$} (M);
    \draw[->] (M) -- (sum4);
    \draw[->] (T1_11) -- (sum3);
    \draw[-] (R) -- (T1_11);
    \draw[->] (T1_21) -- (R);
    \draw[-] (T1_22) -- (T1_21);
    \draw[-] (sum4) -- (Du);
    \draw[-] (Du) -- (Du1);
    \draw[-] (Du1) -- (minD22);
    \draw[-] (minD22) -- (Du3);
    \draw[-] (Du3) -- (Du2);
    \draw[->] (Du2) -- (Dn);
    \draw[->] (sum4) -- (sum5);
    \draw[->,anchor=west] (sum5) -- node{$\mathbf{u}$} (TTyy);
    \draw[->] (DelU) -- node{$\bm\delta_u$} (sum5);
    \draw[-,anchor=east] (sum1) -- node{$\mathbf{y}$} (Dn);

    \node [input,below of=Dn,node distance=1.9cm] (T3) {};
    \node [block,right of=T3,node distance=3.3cm] (L) {$\mathbf{L}$};
    \node [input,right of=L,node distance=3cm] (T4) {};
    \draw[-] (Dn) -- (T3);
    \draw[->] (T3) -- (L);
    \draw[-] (L) -- (T4);
    \draw[->] (T4) -- (sum4);
        
\end{tikzpicture}
     \caption{Output feedback controller structure}
    \label{fig:SyS1_ProperPlant}
\end{figure}

{\em Proper plant $\mathbf{G}$:} The extension to the case of proper $\mathbf{G}$ is straightforward (see \cite[Section~III-D]{wang2019system}). In this case, the dynamics of the open-loop plant $\mathbf{P}$ are given by
\begin{subequations}
\begin{align}
x[t+1] &= Ax[t]+Bu[t]+B_1w[t],\\
y[t] &= Cx[t] + Du[t] + D_1w[t].
\end{align}
\end{subequations}
Define a new measurement $\hat{y}[t]=y[t]-Du[t]$ for this case, which results in the controller structure illustrated in Figure \ref{fig:SyS1_ProperPlant}, and internally stabilizes the plant $\mathbf{G}$. In other words, the controller that internally stabilizes a proper plant $\mathbf{G}$ is given by
\begin{align}\label{eq: K_PrpG}
\begin{split}
&\mathbf{K} = \check{\mathbf{K}}\big(I+D\check{\mathbf{K}}\big)^{-1},\\ &\text{where }\check{\mathbf{K}} = (\mathbf{L}-\mathbf{M}\mathbf{R}^{-1}\mathbf{N}).
\end{split}
\end{align}

In the following section, we build upon the above parameterization of a stabilizing controller, and present a D-SLP for all the plants $\mathbf{G}$ that are stabilized by a given output-feedback controller $\mathbf{K}$.

\section{Dual System-Level Parameterization}\label{sec: D_SLP}
Consider the open-loop system $\mathbf{Q} = [\mathbf{Q}_1~~\mathbf{Q}_2]$ given by
\begin{align}\label{eq: dualPlant}
\bar{\mathbf{u}} = \mathbf{Q}_1\mathbf{w} + \mathbf{Q}_2\bar{\mathbf{y}},
\end{align}
where $\mathbf{Q}_1$ filters the exogenous signal $\mathbf{w}$, and $\mathbf{Q}_2$ filters $\bar{\mathbf{y}}$. Drawing analogies with Figure \ref{fig:sysID_Rep}, $\bar{\mathbf{u}}$ and $\bar{\mathbf{y}}$ are as illustrated in the figure, $\mathbf{w}^\top=[\mathbf{e}^\top~~\mathbf{r}_1^\top~~\mathbf{r}_2^\top]$, $\mathbf{Q}_1=[\mathbf{KS}~~-\mathbf{K}~~-I]$, and $\mathbf{Q}_2=\mathbf{K}$. Analogous to the case in Section \ref{sec: SLP}, we begin by considering a strictly proper $\mathbf{Q}_2$ (or, strictly proper $\mathbf{K}$). The dynamics of the open-loop system $\mathbf{Q}$ under this assumption are given by
\begin{subequations}\label{eq: CtrlStSpace}
\begin{align}
\xi[t+1] &= A_k \xi[t] + B_k \bar{y}[t] + B_{1k}w[t],\\
\bar{u}[t] &= C_k\xi[t] + D_{1k}w[t].
\end{align}
\end{subequations}
Let $\delta_{\xi}[t]:=B_{1k}w[t]$ denote the disturbance to the state $\xi$, and $\delta_{\bar{u}}[t]=D_{1k}w[t]$ be the disturbance to the output $\bar{u}$. For the plant model $\bar{\mathbf{y}}=\mathbf{G}\bar{\mathbf{u}}$, the closed-loop dual system responses (or the dual system-level parameters) are defined as the transfer functions $\{\mathbf{R_k,M_k,N_k,L_k}\}$ that map the disturbance pair $({\bm\delta}_{\xi},{\bm\delta}_{\bar{u}})$ to the state $\bm\xi$ and $\bar{\mathbf{y}}$ as follows
\begin{align}\label{eq: dualSLP_Def}
\begin{bmatrix}
\bm\xi\\
{\mathbf{\bar{y}}}
\end{bmatrix}=
\begin{bmatrix}
\mathbf{R}_k & \mathbf{N}_k\\
\mathbf{M}_k & \mathbf{L}_k
\end{bmatrix}
\begin{bmatrix}
{\bm\delta}_{\xi}\\
{\bm\delta}_{\bar{u}}
\end{bmatrix}.
\end{align}

\begin{figure}
 \centering	
 \begin{tikzpicture}[auto,>=latex']
    \node [input,name=Dy] {};
    \node [sum,right of=Dy,node distance=0.8cm] (sum1) {};
    \node [block,right of=sum1,node distance=1.25cm] (C2){$C_k$};
    \node [input,right of=C2,node distance=1cm] (T1){};
    \node [block,right of=T1,node distance=1cm] (bYz){$1/z$};
    \node [input,below of=T1,node distance=1cm] (T1_1){};
    \node [block,right of=T1_1,node distance=1cm] (A){$A_k$};
    \node [input,right of=A,node distance=1cm] (T1_2){};
    \node [sum,right of=bYz,node distance=1cm] (sum2) {};
    \node [input,above of=sum2,node distance=0.8cm] (Dx) {};
    \node [block,right of=sum2,node distance=1cm] (B2) {$B_k$};
    \node [input,right of=B2,node distance=1cm] (T2) {};
    \draw[->] (Dy)-- node{$\bm\delta_{\bar{u}}$} (sum1);
    \draw[->] (C2) -- (sum1);
    \draw[->,anchor=south] (bYz) -- node{$\bm\xi$} (C2);
    \draw[->] (sum2) -- (bYz);
    \draw[->] (Dx) -- node{$\bm\delta_{\xi}$} (sum2);
    \draw[->] (B2) -- (sum2);
    \draw[->] (T2) -- (B2);
    \draw[-] (T1) -- (T1_1);
    \draw[->] (T1_1) -- (A);
    \draw[-] (A) -- (T1_2);
    \draw[->] (T1_2) -- (sum2);
    
    \node [input,below of=sum1,node distance=2.4cm] (Dn) {};
    \node [block,right of=Dn,node distance=1.3cm] (N) {$\tilde{\mathbf{N}}_k$};
    \node [sum,right of=N,node distance=1cm] (sum3){};
    \node [block,right of=sum3,node distance=1cm] (bYz2){$1/z$};
    \node [input,below of=sum3,node distance=1cm] (T1_11){};
    \node [block,right of=T1_11,node distance=1cm] (R){$\tilde{\mathbf{R}}_k^+$};
    \node [input,right of=R,node distance=1cm] (T1_21){};
    \node [input,right of=bYz2,node distance=1cm] (T1_22) {};
    \node [block,right of=T1_22,node distance=1cm] (M) {$\tilde{\mathbf{M}}_k$};
    \node [sum,right of=M,node distance=1cm] (sum4) {};
    \node [input,right of=sum4,node distance=0.8cm] (Du) {};
    \node [input,above of=sum3,node distance=0.8cm] (Db) {};
    \draw[->] (Db) -- node{$\bm\delta_{\zeta}$} (sum3);
    \draw[->] (Dn)-- (N);
    \draw[->] (N) -- (sum3);
    \draw[->,anchor=south] (sum3) -- (bYz2);
    \draw[->] (bYz2) -- node{$\bm\zeta$} (M);
    \draw[->] (M) -- (sum4);
    \draw[->] (T1_11) -- (sum3);
    \draw[-] (R) -- (T1_11);
    \draw[->] (T1_21) -- (R);
    \draw[-] (T1_22) -- (T1_21);
    \draw[->,anchor=south] (Du) -- node{$\bm\delta_{\bar{y}}$} (sum4);
    \draw[-,anchor=west] (sum4) -- node{$\bar{\mathbf{y}}$} (T2);
    \draw[-,anchor=east] (sum1) -- node{$\bar{\mathbf{u}}$} (Dn);

    \node [input,below of=Dn,node distance=1.9cm] (T3) {};
    \node [block,right of=T3,node distance=3.3cm] (L) {$\mathbf{L}_k$};
    \node [input,right of=L,node distance=3cm] (T4) {};
    \draw[-] (Dn) -- (T3);
    \draw[->] (T3) -- (L);
    \draw[-] (L) -- (T4);
    \draw[->] (T4) -- (sum4);
\end{tikzpicture}
     \caption{Dual closed-loop structure with $\tilde{\mathbf{R}}_k^+=z(I-z\mathbf{R}_k)$, $\tilde{\mathbf{M}}_k = z\mathbf{M}_k$, and $\tilde{\mathbf{N}}_k=z\mathbf{N}_k$. Note that $\tilde{\mathbf{R}}_k^+,\tilde{\mathbf{M}}_k,\tilde{\mathbf{N}}_k\in\mathcal{RH}_{\infty}$.}
    \label{fig:SyS2}
\end{figure}

The following corollary directly follows from Theorem \ref{thm: main}. It algebraically characterizes the dual parameters $\{\mathbf{R_k,M_k,N_k,L_k}\}$ in \eqref{eq: dualSLP_Def}, as well as parameterizes the set of plants $\mathbf{G}$ that are stabilized by the controller $\mathbf{K}$.
\begin{Corollary}
Consider the dynamics \eqref{eq: CtrlStSpace} with the plant model $\bar{\mathbf{y}} = \mathbf{G}\bar{\mathbf{u}}$. Then, the closed-loop dual system-level parameters $\{\mathbf{R}_k,\mathbf{M}_k,\mathbf{N}_k,\mathbf{L}_k\}$ in \eqref{eq: dualSLP_Def} lie in the affine subspace
\begin{subequations}\label{eq: AffSubDual}
\begin{align}
\begin{bmatrix}
zI-A_k & -B_k
\end{bmatrix}
\begin{bmatrix}
\mathbf{R}_k & \mathbf{N}_k\\
\mathbf{M}_k & \mathbf{L}_k
\end{bmatrix}&=\begin{bmatrix}
I & 0
\end{bmatrix},\\
\begin{bmatrix}
\mathbf{R}_k & \mathbf{N}_k\\
\mathbf{M}_k & \mathbf{L}_k
\end{bmatrix}\begin{bmatrix}
zI-A_k\\
-C_k
\end{bmatrix}&=
\begin{bmatrix}
I\\
0
\end{bmatrix},\\
\mathbf{R}_k,\mathbf{M}_k,\mathbf{N}_k\in\frac{1}{z}\mathcal{RH}_{\infty}&,\ \mathbf{L}_k\in\mathcal{RH}_{\infty},
\end{align}
\end{subequations}
and the plant $\hat{\mathbf{G}}=\mathbf{L}_k-\mathbf{M}_k\mathbf{R}_k^{-1}\mathbf{N}_k$ is internally stabilized by the output feedback controller $\mathbf{K}$.
\label{coro:1}
\end{Corollary}

{\em Proper controller $\mathbf{K}$:} The extension to the case of proper controllers $\mathbf{K}$ is straightforward, where the dynamics of the open-loop system $\mathbf{Q}$ are given by
\begin{subequations}
\begin{align}
\xi[t+1] &= A_k \xi[t] + B_k \bar{y}[t] + B_{1k}w[t],\\
\bar{u}[t] &= C_k\xi[t] + D_k \bar{y}[t] + D_{1k}w[t].
\end{align}
\end{subequations}
Analogous to the case of proper plants discussed in Section \ref{sec: SLP}, the plant
\begin{align}\label{eq: EstG_PrpCtrl}
\begin{split}
&\hat{\mathbf{G}} = \check{\mathbf{G}}\big(I+D_k\check{\mathbf{G}}\big)^{-1},\\ &\text{where }\check{\mathbf{G}} = (\mathbf{L}_k-\mathbf{M}_k\mathbf{R}^{-1}_k\mathbf{N}_k),
\end{split}
\end{align}
is stabilized by the proper controller $\mathbf{K}$. 

In the next section, we illustrate the use of the above characterization in terms of the dual system-level parameters in closed-loop identification of the plant $\mathbf{G}$. 

\section{Identification in the Dual-SLS Framework}\label{sec: SysID_D_SLP}
The objective of the D-SLP developed in Section \ref{sec: D_SLP} is to aid the closed-loop identification of the transfer function $\mathbf{G}$. In the proposed identification scheme, we pose an optimization problem to identify the dual parameters $\{\mathbf{R}_k,\mathbf{M}_k,\mathbf{N}_k,\mathbf{L}_k\}$ that conform with the collected closed-loop data, as well as the affine subspace constraints in \eqref{eq: AffSubDual}. More precisely, an open-loop identification equation from the reference signal $\mathbf{r}:=\mathbf{Kr}_1+\mathbf{r}_2$ to the output signal $\mathbf{y}$ is first derived in terms of the dual parameter $\mathbf{L}_k$. Secondly, an optimization problem is posed to estimate the dual parameters $\{\hat{\mathbf{R}}_k,\hat{\mathbf{M}}_k,\hat{\mathbf{N}}_k,\hat{\mathbf{L}}_k\}$ that comply with the above open-loop equation as well as the affine subspace constraints in \eqref{eq: AffSubDual}. Finally, the plant is estimated from Corollary~\ref{coro:1} with $\hat{\mathbf{G}}=\hat{\mathbf{L}}_k-\hat{\mathbf{M}}_k\hat{\mathbf{R}}_k^{-1}\hat{\mathbf{N}}_k$. The following lemma is required to derive the mapping from the reference signal $\mathbf{r}:=\mathbf{Kr}_1+\mathbf{r}_2$ to the output $\mathbf{y}$ in terms of the dual system response $\mathbf{L}_k$.

\begin{Lemma}
For the closed-loop representation illustrated in Figure \ref{fig:sysID_Rep}, where $\bar{\mathbf{y}}=\mathbf{G\bar{u}}$, the dual system response $\mathbf{L}_k=(I-\mathbf{GK})^{-1}\mathbf{G}$.
\end{Lemma}
\begin{proof}
Without loss of generality we consider the controller $\mathbf{K}$ to be strictly proper, i.e., $\mathbf{K} = C_k(zI-A_k)^{-1}B_k$. Note that from Figure \ref{fig:SyS2} we have $\bar{\mathbf{u}}=C_k{\bm\xi}+{\bm\delta}_{\bar{u}}$, which gives
\begin{subequations}
\begin{align}
\bar{\mathbf{y}} &= \mathbf{G}\big(C_k\bm\xi+\bm\delta_{\bar{u}}\big)\\
&= \mathbf{G}C_k(zI-A_k)^{-1}\big(B_k\bar{\mathbf{y}}+\bm\delta_{\xi}\big)+\mathbf{G}\bm\delta_{\bar{u}}\\
&= \mathbf{GK}\bar{\mathbf{y}}+\mathbf{G}C_k(zI-A_k)^{-1}{\bm\delta_{\xi}}+\mathbf{G}{\bm\delta}_{\bar{u},k}.
\end{align}
\end{subequations}
Since $\bar{\mathbf{y}} = \mathbf{M}_k\bm\delta_{\xi} + \mathbf{L}_k\bm\delta_{\bar{u}}$ from \eqref{eq: dualSLP_Def}, we have $\mathbf{M}_k = (I-\mathbf{G}\mathbf{K})^{-1}\mathbf{G}C_k(zI-A_k)^{-1}$ and $\mathbf{L}_k = (I-\mathbf{GK})^{-1}\mathbf{G}$.
\end{proof}

Below we derive the equation that maps the signal $\mathbf{r}:=\mathbf{r}_2+\mathbf{Kr}_1$ to the output $\mathbf{y}$ in terms of $\mathbf{L}_k$. The closed-loop representation in Figure \ref{fig:sysID_Rep} results in the following equation
\begin{subequations}
\begin{align}
\mathbf{y} &= \mathbf{G}\bar{\mathbf{u}}+\mathbf{Se}\\
&= \mathbf{G(Ky+r)+Se}\\
&= (I-\mathbf{GK})^{-1}\mathbf{Gr}+(I-\mathbf{GK})^{-1}\mathbf{Se}\\
&= \mathbf{L}_k\mathbf{r} + (I-\mathbf{GK})^{-1}\mathbf{Se}\label{eq: Lk_SySId}
\end{align}
\end{subequations}

We pose the following optimization problem
\begin{subequations}\label{eq: OptimProb}
\begin{align}
&\min_{\{\hat{\mathbf{R}}_k,\hat{\mathbf{M}}_k,\hat{\mathbf{N}}_k,\hat{\mathbf{L}}_k\}}  g(\mathbf{y},\hat{\mathbf{L}}_k,\mathbf{r})\\
&\text{subject to} 
\begin{bmatrix}
zI-A_k & -B_k
\end{bmatrix}
\begin{bmatrix}
\hat{\mathbf{R}}_k & \hat{\mathbf{N}}_k\\
\hat{\mathbf{M}}_k & \hat{\mathbf{L}}_k
\end{bmatrix}=\begin{bmatrix}
I & 0
\end{bmatrix},\label{eq: OptimProbC1}\\
&\qquad\qquad \begin{bmatrix}
\hat{\mathbf{R}}_k & \hat{\mathbf{N}}_k\\
\hat{\mathbf{M}}_k & \hat{\mathbf{L}}_k
\end{bmatrix}\begin{bmatrix}
zI-A_k\\
-C_k
\end{bmatrix}=
\begin{bmatrix}
I\\
0
\end{bmatrix},\label{eq: OptimProbC2}\\
&\qquad\qquad \hat{\mathbf{R}}_k,\hat{\mathbf{M}}_k,\hat{\mathbf{N}}_k\in\frac{1}{z}\mathcal{RH}_{\infty},\hat{\mathbf{L}}_k\in\mathcal{RH}_{\infty},\label{eq: OptimProbC3}
\end{align}
\end{subequations}
to estimate the dual parameters $\{\mathbf{R}_k,\mathbf{M}_k,\mathbf{N}_k,\mathbf{L}_k\}$, where $g(\cdot)$ is a functional capturing the estimation criteria of $\mathbf{L}_k$ using the open-loop equation \eqref{eq: Lk_SySId}. The estimate $\hat{\mathbf{G}}$ of the transfer function $\mathbf{G}$ is given by $\hat{\mathbf{G}}=\hat{\mathbf{L}}_k-\hat{\mathbf{M}}_k\hat{\mathbf{R}}_k^{-1}\hat{\mathbf{N}}_k$ for strictly proper controllers $\mathbf{K}$, or as described in \eqref{eq: EstG_PrpCtrl} for the case of proper controllers.

\begin{Remark}\label{rem: Consistency}
Note that the signal $\mathbf{r}$ and the noise $\mathbf{e}$ in \eqref{eq: Lk_SySId} are uncorrelated. Let $\mathcal{L}$, $\mathcal{M}$, $\mathcal{N}$, and $\mathcal{R}$ be the model class for the dual parameters comprising of the admissible solutions to \eqref{eq: OptimProbC1}-\eqref{eq: OptimProbC3}.  It is known that \cite{LjungBook2} under the assumption that the model class $\mathcal{L}$, $\mathcal{M}$, $\mathcal{N}$, and $\mathcal{R}$ contains the actual dual parameters $\mathbf{L}_k$, $\mathbf{M}_k$, $\mathbf{N}_k$, and $\mathbf{R}_k$, respectively, then the estimates $\hat{\mathbf{L}}_k$, $\hat{\mathbf{M}}_k$, $\hat{\mathbf{N}}_k$, and $\hat{\mathbf{R}}_k$ can be obtained consistently with an independently parameterized noise model.
\end{Remark}

{\em Invariance to the state-space realization of $\mathbf{K}$}: Note that even though the realization $(A_k, B_k, C_k)$ of the controller $\mathbf{K}$ define the affine constraints in \eqref{eq: OptimProb}, the optimization problem remains invariant to feasible linear transformations of the realization. More precisely, consider the non-singular linear transformation matrix $\bar{T}$ of the state $\xi$ in \eqref{eq: CtrlStSpace}. The corresponding transition matrices are given by $\bar{A}_k=\bar{T}^{-1}A_k\bar{T}$, $\bar{B}_k=\bar{T}^{-1}B_k$, and $\bar{C}_k=C_k\bar{T}$, and the constraint \eqref{eq: OptimProbC1} becomes
\begin{align}
\begin{bmatrix}
zI-\bar{T}^{-1}A_k\bar{T} & -\bar{T}^{-1}B_k
\end{bmatrix}
\begin{bmatrix}
\hat{\mathbf{R}}_k & \hat{\mathbf{N}}_k\\
\hat{\mathbf{M}}_k & \hat{\mathbf{L}}_k
\end{bmatrix}=\begin{bmatrix}
I & 0
\end{bmatrix}.
\end{align}
With simple algebraic manipulations, it can be shown that the above constraint is equivalent to the following
\begin{align}\label{eq: OptimProbC1_2}
\begin{bmatrix}
zI-A_k & -B_k
\end{bmatrix}
\begin{bmatrix}
\bar{T}\hat{\mathbf{R}}_k\bar{T}^{-1} & \bar{T}\hat{\mathbf{N}}_k\\
\hat{\mathbf{M}}_k\bar{T}^{-1} & \hat{\mathbf{L}}_k
\end{bmatrix}=\begin{bmatrix}
I & 0
\end{bmatrix}.
\end{align}
Similarly, for the realization $(\bar{A}_k,\bar{B}_k,\bar{C}_k)$, the constraint \eqref{eq: OptimProbC2} can equivalently be written as 
\begin{align}\label{eq: OptimProbC2_2}
\begin{bmatrix}
\bar{T}\hat{\mathbf{R}}_k\bar{T}^{-1} & \bar{T}\hat{\mathbf{N}}_k\\
\hat{\mathbf{M}}_k\bar{T}^{-1} & \hat{\mathbf{L}}_k
\end{bmatrix}
\begin{bmatrix}
zI-A_k\\
-C_k
\end{bmatrix}=\begin{bmatrix}
I \\
0
\end{bmatrix}.
\end{align}
Note that the dual parameter $\hat{\mathbf{L}}_k$ remains invariant to the above transformation. Thus, the functional $g(\cdot)$, which quantifies the estimation criteria of $\hat{\mathbf{L}}_k$, is invariant. Further, defining $\check{\mathbf{R}}_k:=\bar{T}\hat{\mathbf{R}}_K\bar{T}^{-1}$, $\check{\mathbf{M}}_k:=\hat{\mathbf{M}}_k\bar{T}^{-1}$, $\check{\mathbf{N}}_k:=\bar{T}\hat{\mathbf{N}}_k$, in the constraints \eqref{eq: OptimProbC1_2} and \eqref{eq: OptimProbC2_2}, we recover an equivalent optimization problem to (\ref{eq: OptimProb}). Thus, the identification problem formulation presented \eqref{eq: OptimProb} is independent of the realization of the controller $\mathbf{K}$.

\begin{Remark}
As is evident from \eqref{eq: OptimProb} and the above detail, the D-SLP method does not depend on an initial nominal plant $\mathbf{G}_0$ that is stabilized by the controller $\mathbf{K}$, as well as the realization of the controller $\mathbf{K}$. This is contrary to some of the popular closed-loop identification approaches (such as \cite{schrama1991open}, \cite{hansen1988fractional}) illustrated in the Section \ref{sec: Introduction} which depend on an initial nominal plant $\mathbf{G}_0$, along with its co-prime factors, and the co-prime factors of the controller $\mathbf{K}$. We further illustrate on this upside of D-SLP approach using simulations in the following section.
\end{Remark}

\begin{figure*}
    \centering
    \includegraphics[width=\linewidth]{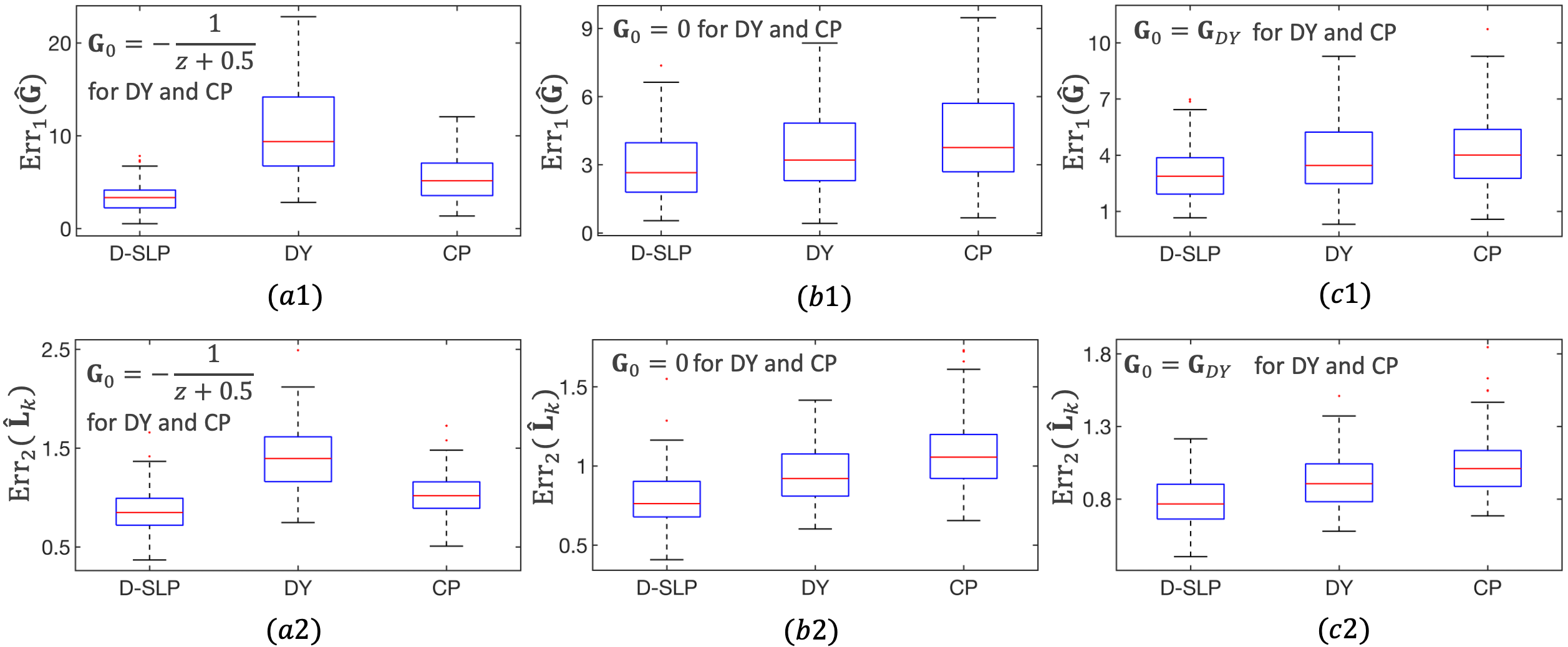}
    \caption{Monte Carlo simulations for the closed-loop system in \eqref{eq: VanDanEx}. Each box plots shows the minimum, first quantile, median, third quantile, and the maximum. Each simulation has a different noise realization. The choice of nominal plant $\mathbf{G}_0$ for the dual-Youla (DY) and the co-prime (CP) methods is mentioned in each figure. D-SLP corresponds to the dual system-level parameterization method. (a1)-(a2) The DY and CP methods do not perform well for the given choice of $\mathbf{G}_0$ in (a), and D-SLP outperforms both. (b1)-(b2) and (c1)-(c2) The choice of nominal plant $\mathbf{G}_0$ in (b) and (c) show improvements for the DY and CP method, however, the performance of D-SLP is still better than the former methods.}
    \label{fig: Simulation1}
\end{figure*}

\section{Simulations}\label{sec: Simulations}
In this section, the D-SLP identification scheme is compared with the existing benchmark approaches: a) the co-prime identification scheme \cite{schrama1991open} and b) the dual-Youla method \cite{hansen1988fractional}. Note that, by design, the dual-Youla method and our proposed D-SLP framework, guarantee that the estimated transfer function $\hat{\mathbf{G}}$ is stabilized by the known controller $\mathbf{K}$. However, our proposed D-SLP method is independent of the several hyperparameters that are otherwise required by the co-prime factorization and the dual-Youla methods. Below we provide very concise illustrations of the latter approaches. Please refer to \cite{hansen1988fractional} and \cite{schrama1991open} for more details on the individual methods.

\subsection{Brief Overview of the Benchmark Methods}
Let $\mathbf{G}$ be the transfer function to be identified in the closed-loop representation in Figure \ref{fig:sysID_Rep}, and $\mathbf{K}$ be its stabilizing controller. Let $\mathbf{G}_0$ be a nominal plant that is stabilized by $\mathbf{K}$, with $\mathbf{N}_0$ and $\mathbf{D}_0$ as its stable co-prime factors. The dual-Youla identification involves the following steps.
\begin{itemize}
\item[(a)] Let $\mathbf{X}_0$ and $\mathbf{Y}_0$ be the stable co-prime factors of $\mathbf{K}$.
\item[(b)] Determine the signals $\bm\beta:=\mathbf{D}_0\mathbf{y}-\mathbf{N}_0\mathbf{u}$ and $\bm\alpha=\mathbf{Y}_0\mathbf{r}$.
\item[(c)] Estimate the Youla parameter $\mathbf{R}\in\mathcal{RH}_{\infty}$, using the open-loop equation
\begin{align}\label{eq: OpenLoopDY}
\bm\beta = \mathbf{R}\bm\alpha+\mathbf{Fe}.
\end{align}
\item[(d)] The $\mathbf{G}$ estimate is given by $\hat{\mathbf{G}}=\frac{\mathbf{N}_0+\hat{\mathbf{R}}\mathbf{Y}_0}{\mathbf{D}_0-\hat{\mathbf{R}}\mathbf{X}_0}$, where $\hat{\mathbf{R}}$ is the estimate of $\mathbf{R}$ in \eqref{eq: OpenLoopDY}.
\end{itemize}
The steps involved in the co-prime factorization method are
\begin{itemize}
\item[(a)] Let $\mathbf{N}\in\mathcal{RH}_{\infty}$ and $\mathbf{D}\in\mathcal{RH}_{\infty}$ be the right co-prime factors of $\mathbf{G}$. These co-prime factors are identified in an open-loop way using the following equations
\begin{align}\label{eq: CP_Eqns}
\mathbf{u} = \mathbf{D}\mathbf{x} - \mathbf{X}_0\mathbf{Se},\quad
\mathbf{y} = \mathbf{N}\mathbf{x} + \mathbf{Y}_0\mathbf{Se},
\end{align}
where the signal $\mathbf{x} = (\mathbf{D}_0+\mathbf{CN}_0)^{-1}\mathbf{r}$ in the context of the representation in Figure \ref{fig:sysID_Rep}, and $\mathbf{X}_0$ and $\mathbf{Y}_0$ are the right co-prime factors of the controller $\mathbf{K}$.
\item[(b)] The $\mathbf{G}$ estimate is given by $\hat{\mathbf{G}}=\hat{\mathbf{N}}\big(\hat{\mathbf{D}}\big)^{-1}$, where $\hat{\mathbf{N}}$ and $\hat{\mathbf{D}}$ are the estimates of $\mathbf{N}$ and $\mathbf{D}$, respectively, in \eqref{eq: CP_Eqns}.
\end{itemize}

Note that both the co-prime factorization and the dual-Youla parameterization methods are dependent on (i) the choice of a nominal plant $\mathbf{G}_0$ that is stabilized by the controller $\mathbf{K}$, and (ii) the co-prime factors $\mathbf{N}_0$, $\mathbf{D}_0$ of $\mathbf{G}_0$. Further, the signal $\bm\alpha$ in the dual-Youla method is obtained by filtering $\mathbf{r}$ through $\mathbf{Y}_0$, thus making it dependent on the choice of the co-prime factors $\mathbf{X}_0, \mathbf{Y}_0$ of $\mathbf{K}$. On the contrary, the D-SLP method requires neither a nominal plant $\mathbf{G}_0$, nor its co-prime factorization and the co-prime factors of the controller. In other words, the benefit of the D-SLP identification scheme presented in Section \ref{sec: SysID_D_SLP} is independent of the choice of hyper-parameters that are otherwise required in \cite{schrama1991open} and \cite{hansen1988fractional}. 

\subsection{Choice of Model Structure}\label{sec: Remark1}
Similar to what is presented in \cite{wang2019system}, under the dual system-level parameterization of the transfer function $\mathbf{G}$, one of the convenient ways to solve the optimization problem \eqref{eq: OptimProb} is under the assumption that the dual parameters $\{\hat{\mathbf{R}}_k,\hat{\mathbf{M}}_k,\hat{\mathbf{N}}_k,\hat{\mathbf{L}}_k\}$ are finite impulse responses (FIR). For instance, the subspace constraint $\hat{\mathbf{R}}_k\in\frac{1}{z}\mathcal{RH}_{\infty}$ is easily satisfied for the FIR case $\hat{\mathbf{R}}_k=\sum_{i=0}^T z^{-i-1}\hat{R}_k[i]$, where $T$ denotes the horizon of the impulse response, and $\hat{R}_k[i]$ for all $0\leq i\leq T$ are the impulse response elements. Further, the affine constraints in the optimization problem \eqref{eq: OptimProb} translate to affine constraints on the FIR coefficients $\{\hat{R}_k[t],\hat{M}_k[t],\hat{N}_k[t],\hat{L}_k[t]\}_{i=0}^T$. It is shown in \cite{wang2019system} that for a controllable and observable system $(A_k,B_k,C_k)$ such FIR representations of the dual parameters are always feasible for suitably chosen horizons $T$. When the system $(A_k,B_k,C_k)$ is only stabilizable and/or detectable such that the FIR feasibility cannot be satisfied, the relaxation similar to those on the system-level parameters presented in \cite{8264168} can be used. Note that similar FIR modelling assumptions are also applicable to other closed-loop identification techniques. For instance, the sensitivity function identified in the indirect method \cite{van1993indirect}, the Youla parameter in \cite{hansen1988fractional}, and the co-prime factors identified in \cite{schrama1991open} are constrained to lie in $\mathcal{RH}_{\infty}$.

\subsection{Numerical Examples}
Consider the example setting as illustrated in \cite{van1993indirect}, where
\begin{align}\label{eq: VanDanEx}
\begin{split}
\mathbf{G} &= \frac{z^2}{z^2-1.6z+0.89},\quad\mathbf{K} = \frac{z-0.8}{z^2},\\
\mathbf{S} &= \frac{z^3-1.56z^2+1.045z-0.3338}{z^3-2.35z^2+2.09z-0.6675}.
\end{split}
\end{align}
The external noise vector $e[t]$ is sampled from a zero-mean Gaussian distribution $\mathcal{N}(0,\gamma^2)$. For the purpose of simulations, we consider $\mathbf{r}_1$ to be 0, $\mathbf{r}_2$ to be a periodic pseudorandom binary sequence (PRBS) of magnitude $\{+10,-10\}$ with a period $m=10$ and each period of length $n=2^9-1$, and $\gamma=2$. We also consider FIR model structure (with a horizon $T=15$) for the dual system-level parameters, the Youla parameters, and the co-prime factors as discussed in Section \ref{sec: Remark1}. The impulse response elements in the D-SLP method are determined by minimizing the cost $g=\|y_0^{\bar{n}}-\Phi(r_0^{\bar{n}})L_k^T\|_2^2$ subject to constraints in \eqref{eq: OptimProb}, where $y_0^{\bar{n}}=[y_0~y_1~\cdots~y_{\bar{n}}]^\top\in\mathbb{R}^{\bar{n}}$ denotes the output data trajectory of length $\bar{n}=nm$, $r_0^{\bar{n}}=[r_0~r_1~\cdots~r_{\bar{n}}]\in\mathbb{R}^{\bar{n}}$ denotes the input data trajectory of length $\bar{n}$, $\Phi(r)\in\mathbb{R}^{\bar{n}\times (T+1)}$ is the toeplitz matrix with first column as $r_0^{\bar{n}}$ and first row as $[r_0~0\cdots~0]\in\mathbb{R}^{T+1}$, and $L_k^T=[L_k[0]~L_k[1]~\cdots~L_k[T+1]]^\top$ is a vector of impulse response elements of $\mathbf{L}_k$. Note that this choice of $g$ fits the data to the input-output relationship $y_k=\sum_{i=0}^{\bar{n}}L_k[i]r_{k-i} + v(k)$, for all $k\in\{0,1,\cdots,T\}$, where the input $r_k=0$ for all $k<0$ assuming the system at rest, and $v(k)$ is the filtered noise. Similar, input-output relationship and cost functions are chosen in the identification of $\mathbf{R}$ in \eqref{eq: OpenLoopDY} (dual-Youla), and $\mathbf{N},\mathbf{D}$ in \eqref{eq: CP_Eqns} (co-prime factorization scheme).

To assess the performance of each of the identification schemes, 100 Monte Carlo simulations with different noise realizations are performed. It is shown that the D-SLP approach performs better than the dual-Youla (DY) and the co-prime (CP) factorization methods on the closed-loop system specified in \eqref{eq: VanDanEx}, for different choices of the nominal plant $\mathbf{G}_0$ that is required for the DY and CP methods. More precisely, we consider the following three choices for $\mathbf{G}_0$:
\begin{itemize}
\item[(a)] $\mathbf{G}_0=-\frac{1}{z+0.5}$ - arbitrarily selected and stabilized by $\mathbf{K}$.
\item[(b)] $\mathbf{G}_0=0$ - the resulting closed-loop is a zero transfer function, hence stable.
\item[(c)] $\mathbf{G}_0=\hat{\mathbf{G}}_{DY}$ - the estimate of the transfer function $\mathbf{G}$ as given by the dual-Youla method for the case (b) above. This choice is analogous to a two-stage approach in which the dual-Youla method improves upon its estimate of $\hat{\mathbf{G}}$ from the first stage (i.e., when the choice of $\mathbf{G}_0=0$), using additional data.
\end{itemize} 
The performance of the different identification schemes is quantified by the two metrics below
\begin{subequations}\label{eq: PerfMetric2}
\begin{align}
{\tt Err}_1 (\hat{\mathbf{G}}) &= \sum_{i=1}^n 100\frac{\big\|\mathbf{G}(j\omega_i)-\hat{\mathbf{G}}(j\omega_i)\big\|_2}{\big\|\mathbf{G}(j\omega_i)\big\|_2}\\
{\tt Err}_2 (\hat{\mathbf{L}}_k) &= \sum_{i=1}^n 100\frac{\big\|\mathbf{L}_k(j\omega_i)-\hat{\mathbf{L}}_k(j\omega_i)\big\|_2}{\big\|\mathbf{L}_k(j\omega_i)\big\|_2}\\
\text{where }\mathbf{L}_k(j\omega_i) &= \big(I-\mathbf{G}(j\omega_i)\mathbf{K}(j\omega_i)\big)^{-1}\mathbf{G}(j\omega_i), \\
\hat{\mathbf{L}}_k(j\omega_i) &= \big(I-\hat{\mathbf{G}}(j\omega_i)\mathbf{K}(j\omega_i)\big)^{-1}\hat{\mathbf{G}}(j\omega_i),
\end{align}
\end{subequations}
$\hat{\mathbf{G}}$ is the estimate of the transfer function $\mathbf{G}$, $n$ is the length of the signal $\mathbf{r}$, and $\omega_i$'s are $n$ equally spaced frequencies in the range $[0,\pi]$. Note that ${\tt Err}_1(\hat{\mathbf{G}})$ captures the error in the estimate of the plant $\mathbf{G}$, and ${\tt Err}_2(\hat{\mathbf{L}}_k)$  captures the error in the closed-loop plant resulting from the estimate $\hat{\mathbf{G}}$. As illustrated in Figure \ref{fig: Simulation1}, the median of ${\tt Err}_1(\hat{\mathbf{G}})$ and ${\tt Err}_2(\hat{\mathbf{L}}_k)$ is the least for the D-SLP scheme for all the above three choices of the nominal plant $\mathbf{G}_0$ in the DY and CP methods. The same observation holds true regarding the spread of the boxplots of ${\tt Err}_1(\hat{\mathbf{G}})$ and ${\tt Err}_2(\hat{\mathbf{L}}_k)$ as shown in Figure \ref{fig: Simulation1}. Note that the performance of D-SLP method remains the same across Figures \ref{fig: Simulation1}(a), \ref{fig: Simulation1}(b), and \ref{fig: Simulation1}(c), as it is independent of the choice of nominal plant $\mathbf{G}_0$; however, very minute difference arises due to different noise realization in each experimental run performed for the simulations in Figures \ref{fig: Simulation1}(a), \ref{fig: Simulation1}(b), and \ref{fig: Simulation1}(c).

Similar observations as above are seen for the case of the proper controller $\mathbf{K}=\frac{z^2-0.8z}{z^2}$ that stabilizes the plant $\mathbf{G}$ in \eqref{eq: VanDanEx}. Note that here the estimate $\hat{\mathbf{G}}$ of the plant is given by \eqref{eq: EstG_PrpCtrl}. Finally, Figure \ref{fig: Simulation2} demonstrates the asymptotical convergence of the error ${\tt Err}_1(\hat{\mathbf{G}})$ in the estimated plant. This observation is in-line with the Remark \ref{rem: Consistency} on consistency of the D-SLP method in Section \ref{sec: D_SLP}.

\begin{figure}
    \centering
    \includegraphics[width=0.95\columnwidth]{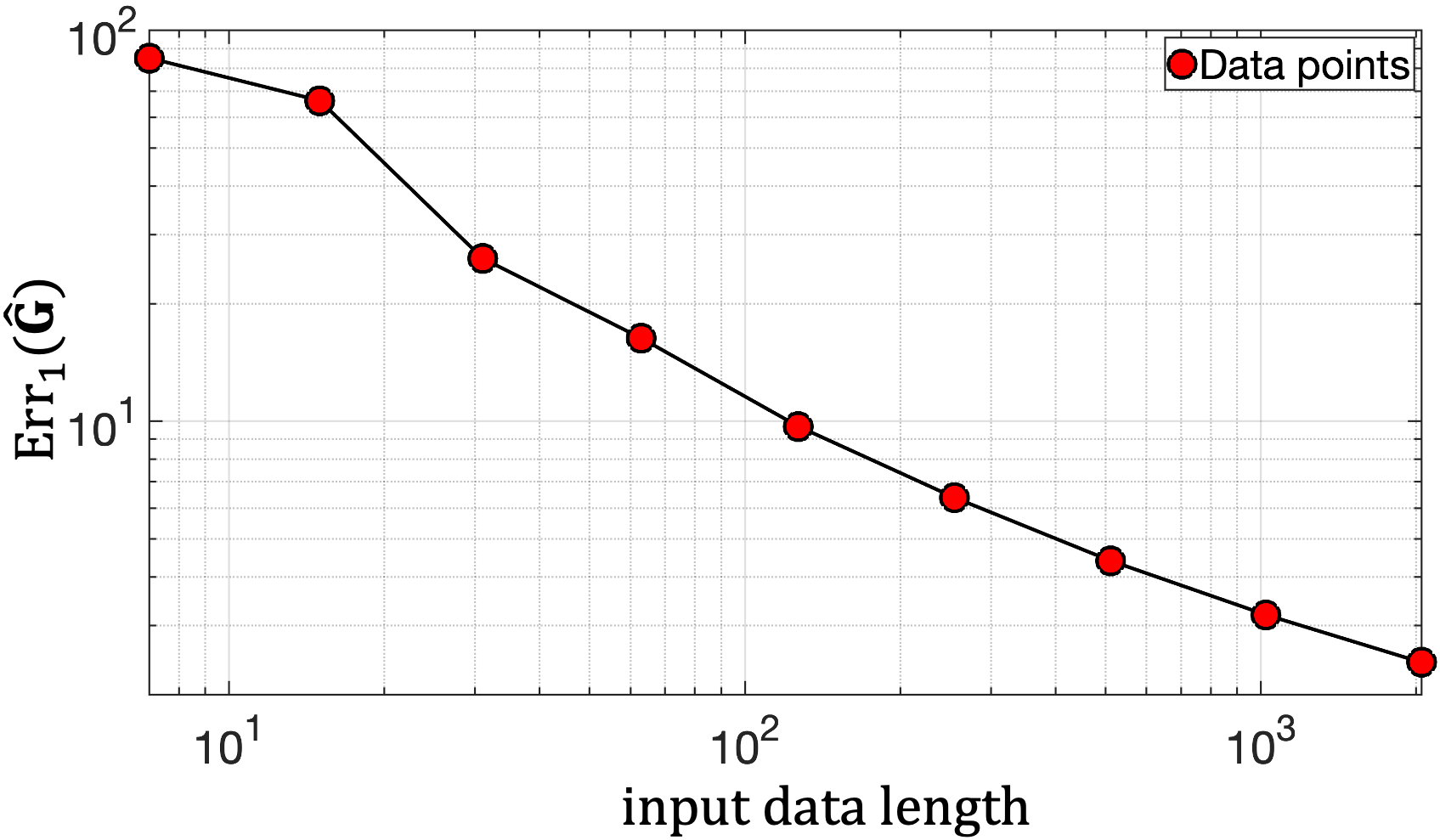}
    \caption{Plot of ${\tt Err}_1(\mathbf{G})$ versus input data length $n$; shows asymptotic convergence of error.}
    \label{fig: Simulation2}
\end{figure}

\section{Conclusion}\label{sec: Conclusions}
In this work, we develop an identification methodology using closed-loop data which uses the dual system level parameterization of plants that are stabilized by a given controller. We demonstrate its capabilities in terms of a) giving a plant estimate that is stabilized by the controller, b) being independent of the choice of controller realizations and several hyper-parameters (such as a nominal plant, its co-prime factors, and the factorization of the controller) that occur in different benchmark methods, and c) outperforming the benchmark methods in terms of the plant and closed-loop identification errors. Just as the system level parameterization \cite{wang2019system} is useful for controller design in the large-scale decentralized system, the utility of proposed D-SLP framework in identification for large scale networks forms the future research direction for this work.

\bibliographystyle{IEEEtran}
\bibliography{biblio_SysID}

\begin{thebibliography}{10}
\providecommand{\url}[1]{#1}
\csname url@samestyle\endcsname
\providecommand{\newblock}{\relax}
\providecommand{\bibinfo}[2]{#2}
\providecommand{\BIBentrySTDinterwordspacing}{\spaceskip=0pt\relax}
\providecommand{\BIBentryALTinterwordstretchfactor}{4}
\providecommand{\BIBentryALTinterwordspacing}{\spaceskip=\fontdimen2\font plus
\BIBentryALTinterwordstretchfactor\fontdimen3\font minus
  \fontdimen4\font\relax}
\providecommand{\BIBforeignlanguage}[2]{{%
\expandafter\ifx\csname l@#1\endcsname\relax
\typeout{** WARNING: IEEEtran.bst: No hyphenation pattern has been}%
\typeout{** loaded for the language `#1'. Using the pattern for}%
\typeout{** the default language instead.}%
\else
\language=\csname l@#1\endcsname
\fi
#2}}
\providecommand{\BIBdecl}{\relax}
\BIBdecl

\bibitem{LjungBook2}
L.~Ljung, \emph{System Identification: Theory for the User}.\hskip 1em plus
  0.5em minus 0.4em\relax Upper Saddle River, NJ, USA: Prentice-Hall, 1999.

\bibitem{Gustavsson_1977}
I.~Gustavsson, L.~Ljung, and T.~Söderström, ``Identification of processes in
  closed loop{\textemdash}identifiability and accuracy aspects,''
  \emph{Automatica}, vol.~13, no.~1, pp. 59--75, 1977.

\bibitem{VANDENHOF1998173}
P.~{Van den Hof}, ``Closed-loop issues in system identification,'' \emph{Annual
  Reviews in Control}, vol.~22, pp. 173--186, 1998.

\bibitem{van1993indirect}
P.~M. Van~den Hof and R.~J. Schrama, ``An indirect method for transfer function
  estimation from closed loop data,'' \emph{Automatica}, vol.~29, no.~6, pp.
  1523--1527, 1993.

\bibitem{schrama1991open}
R.~J. Schrama, ``An open-loop solution to the approximate closed-loop
  identification problem,'' \emph{IFAC Proceedings Volumes}, vol.~24, no.~3,
  pp. 761--766, 1991.

\bibitem{Van_den_Hof_1995}
P.~M. {Van den Hof}, R.~J. Schrama, R.~A. de~Callafon, and O.~H. Bosgra,
  ``Identification of normalised coprime plant factors from closed-loop
  experimental data,'' \emph{European Journal of Control}, vol.~1, no.~1, pp.
  62--74, 1995.

\bibitem{hansen1988fractional}
F.~R. Hansen and G.~F. Franklin, ``On a fractional representation approach to
  closed-loop experiment design,'' in \emph{1988 American Control
  Conference}.\hskip 1em plus 0.5em minus 0.4em\relax IEEE, 1988, pp.
  1319--1320.

\bibitem{4790411}
F.~Hansen, G.~Franklin, and R.~Kosut, ``Closed-loop identification via the
  fractional representation: Experiment design,'' in \emph{1989 American
  Control Conference}, 1989, pp. 1422--1427.

\bibitem{skogestad2005multivariable}
S.~Skogestad and I.~Postlethwaite, \emph{Multivariable Feedback Control:
  Analysis and Design}.\hskip 1em plus 0.5em minus 0.4em\relax Chichester,
  England: John Wiley \& Sons, 2005.

\bibitem{wang2019system}
Y.-S. Wang, N.~Matni, and J.~C. Doyle, ``A system-level approach to controller
  synthesis,'' \emph{IEEE Transactions on Automatic Control}, vol.~64, no.~10,
  pp. 4079--4093, 2019.

\bibitem{wang2018separable}
------, ``Separable and localized system-level synthesis for large-scale
  systems,'' \emph{IEEE Transactions on Automatic Control}, vol.~63, no.~12,
  pp. 4234--4249, 2018.

\bibitem{8264168}
N.~Matni, Y.-S. Wang, and J.~Anderson, ``Scalable system level synthesis for
  virtually localizable systems,'' in \emph{IEEE Conference on Decision and
  Control (CDC)}, 2017, pp. 3473--3480.

\end{thebibliography}
\end{document}